\documentclass[10pt]{amsart}
\usepackage[margin=1.2in,marginparsep=0.1in,marginparwidth=1in]{geometry}
\usepackage{amssymb,amsmath,amsthm,amstext,amscd,latexsym,graphics,graphicx,bbm,caption}
\usepackage[usenames,dvipsnames,svgnames,table]{xcolor}
\usepackage[plainpages=false,colorlinks=true, pagebackref]{hyperref}
\usepackage{tikz}
\usepackage{tikz-cd}
\usetikzlibrary{positioning}
\tikzset{>=stealth}
\hypersetup{citecolor=Sepia,linkcolor=blue, urlcolor=blue}

\usepackage{cite}   

\makeatletter
\def\@tocline#1#2#3#4#5#6#7{\relax
  \ifnum #1>\c@tocdepth 
  \else
    \par \addpenalty\@secpenalty\addvspace{#2}%
    \begingroup \hyphenpenalty\@M
    \@ifempty{#4}{%
      \@tempdima\csname r@tocindent\number#1\endcsname\relax
    }{%
      \@tempdima#4\relax
    }%
    \parindent\z@ \leftskip#3\relax \advance\leftskip\@tempdima\relax
    \rightskip\@pnumwidth plus4em \parfillskip-\@pnumwidth
    #5\leavevmode\hskip-\@tempdima
      \ifcase #1
       \or\or \hskip 2em \or \hskip 2em \else \hskip 3em \fi%
      #6\nobreak\relax
    \dotfill\hbox to\@pnumwidth{\@tocpagenum{#7}}\par
    \nobreak
    \endgroup
  \fi}
\makeatother
\newtheorem{intro-thm}{Theorem}[]
\theoremstyle{plain}
\newtheorem{thm}{Theorem}[section]
\newtheorem{theorem}[thm]{Theorem}

\newtheorem{lemma}[thm]{Lemma}

\theoremstyle{definition}
\newtheorem{remark}[thm]{Remark}
\newtheorem{point}[thm]{}

\newtheorem{definition}[thm]{Definition}

\newcommand{\ilim}{\mathop{\varprojlim}\limits} 

\newcommand{\Union}{\bigcup}

\newcommand{\Ker}{{\rm Ker  }}

\renewcommand{\tilde}{\widetilde}

\newcommand{\N}{{\mathbb N}}

\newcommand{\Z}{{\mathbb Z}}

\newcommand{\ab}[1]{\langle {#1} \rangle}
\newcommand{\ac}[1]{\{#1\}}

\newcommand{\Com}{{\mathsf{ComRings}}}

\input{xy}
\xyoption{all}
\begin{document}

\title{A universal group-theoretic characterisation of $p$-typical Witt vectors} 

\author[S. Pisolkar]{Supriya Pisolkar} \address{ Indian Institute of
Science, Education and Research (IISER),  Homi Bhabha Road, Pashan,
Pune - 411008, India} \email{supriya@iiserpune.ac.in} 

\author[B. Samanta] {Biswanath Samanta} \address{ Indian Institute of
Science, Education and Research (IISER),  Homi Bhabha Road, Pashan,
Pune - 411008, India} \email{biswanath.samanta@students.iiserpune.ac.in} 

\thanks{The second-named author was supported by Ph.D. fellowship (File No: 09/936(0315)/2021-
EMR-I) of Council of Scientific \& Industrial Research (CSIR), India.}

\begin{abstract} 
For a prime $p$ and a commutative ring $R$ with unity, let $W(R)$ denote the group of $p$-typical Witt vectors. The group $W(R)$ is endowed with a Verschiebung operator $W(R)\xrightarrow{V}W(R)$ and a Teichm\"{u}ller map $R\xrightarrow{\ab{ \ }}W(R)$. One of the properties satisfied by $V, \ab{ \ }$ is that the map $R \to W(R)$ given by $x\mapsto V\ab{x^p} - p\ab{x}$ is an additive map. In this paper we show that for $p\neq 2$, this property essentially characterises the functor $W$ (Theorem \ref{univcom}). Unlike other characterisations (see \cite{b}, \cite{J}), this is a group-theoretic characterisation, in the sense that it does not use the ring structure of $W(R)$. Most constructions of the group of $p$-typical Witt vectors of non-commutative rings do not have a ring structure, and hence the above characterisation is more suitable for generalisation to the non-commutative setup.
\end{abstract}
\maketitle 
\section{Introduction}
\noindent Fix a prime $p$ throughout this article. Let  $\Com$ denote the category of Commutative rings with unity and $\mathsf{Ab}$ will denote the category of abelian groups. On the category $\Com$ we have a classical functor of $p$-typical Witt vectors which assigns to every commutative ring $R$, a ring $W(R)$. 

\noindent While $W$ admits an explicit construction \cite{w} using Witt polynomials, there is also a  characterisation of $W$ as a functor satisfying the following universal property: $W$ is the right adjoint of the forgetful functor from the category of $\delta$-rings to the category $\Com$ (see \cite[Theorem 2]{J}).\\

\noindent In this paper we will only be interested in the group structure on $W(R)$, hence we will consider the functor $W$ as taking values in $\mathsf{Ab}$,
$$W : \Com \to \mathsf{Ab}$$

\noindent For a non-commutative ring $R$, the object $W(R)$ generally forms only an abelian group (see \cite{dkp}, \cite{h1}, \cite{h2}). Consequently, in this manuscript, although all rings under consideration are commutative, the ring structure on $W(R)$ is deliberately avoided. This approach facilitates a characterization of Witt vectors that can be generalized to the non-commutative setting. This work is part of an ongoing project and will be addressed in detail in a sequel to this manuscript.
 
\begin{point}\label{properties}
\noindent  For a commutative ring $R$ there exist functorial maps
\begin{enumerate}
\item[-] (Verschiebung) $V: W(R)\to W(R)$ (group homomorphism)
\item[-] (Teichm\"{u}ller map) $\langle - \rangle: R \to W(R)$ (set map).
\end{enumerate} 
which satisfy the following properties

\begin{enumerate} 
\item $\ab{0}=0$. If $p\neq 2$ then $\ab{-x} = -\ab{x}$ for all $x \in R$.
\item $x \mapsto V\ab{x^p} -p\ab{x}$ is an additive map from $R \to W(R)$. 
\item $W(R)$ is complete w.r.t to the filtration $\{V^nW(R)\ | \ n \in \N_0 \}$.
\item $A$ is $p$-torsion free $\implies W(A)$ is $p$-torsion free (see Lemma \eqref{property4}).
\end{enumerate}

\end{point}

\begin{remark}The above properties do not use the ring structure on $W(R)$. For e.g. 
the Teichm\"{u}ller map  $R \xrightarrow{\ab{\ }} W(R)$ is multiplicative. However, since we intend not to use the ring structure of $W(R)$, only a weaker consequence of multiplicativity is stated in the form of (1). 
\end{remark}
\noindent We now abstract the properties \ref{properties} in the following definition.

\begin{definition}[pre-Witt functor]
A functor $F:\Com \to \mathsf{Ab}$ is called a pre-Witt functor, if for every $R$, $F(R)$ is endowed with a functorial endomorphism $F(R)\xrightarrow{V}F(R)$ and a functorial map  $R\xrightarrow{\ab{\ }} F(R)$ such that the properties \ref{properties} are satisfied.
\end{definition}
\begin{remark} $R\mapsto W(R)$ is an example of a pre-Witt functor. 
\end{remark}

\begin{remark}
For every ring $R$, $W(R)$ admits one more functorial endomorphism, viz., the Frobenius endomorphism $W(R)\xrightarrow{\bf{F}}W(R)$. However, $\bf{F}$ is uniquely determined by the requirements ${\bf F}\ab{x}=\ab{x^p}$ and ${\bf F}V=p$.  
\end{remark}

\noindent The following theorem, which is the main result of this paper, shows that the properties as in  \ref{properties} characterise the functor $W$ on $\Com$ for $p\neq 2$.

\begin{theorem}\label{univcom} For $p\neq 2$, the classical functor of $p$-typical Witt vectors
$W: \Com \to Ab$ is a universal pre-Witt functor.

\end{theorem}
\vspace{2mm}

\noindent The outline of the paper is as follows. In the second section we will recall the definition of the functor $E: \Com \to \mathsf{Ab}$ given by Cuntz and Deninger \cite{cd}. The functor $E$ is essentially an alternative construction of $W$ (see \cite[page 561]{cd}). In the third section, for $p\neq 2$, we give a construction of a functor $C: \Com \to \mathsf{Ab}$ which is obtained by enforcing the properties \ref{properties}. At this stage however, it is not straightforward to check that $C$ satisfies \ref{properties}(4). Hence we introduce a temporary notion of {\it weak  pre-Witt functor}, in which a weaker version of \ref{properties}(4) is used (see Definition \ref{defweakpre}). We will then prove that $C$ is a universal weak pre-Witt functor. The fourth section will be devoted to proving that $C$ and $E$  are isomorphic. This shows that $C$ is a universal pre-Witt functor. This is used to finish the proof of Theorem \ref{univcom}.

\section{The functor $E: \Com \to \mathsf{Ab}$}

\noindent We recall the following definitions and results by Cuntz and Deninger \cite{cd} for the convenience of the reader. Throughout this subsection we will assume that $R \in \Com$.

\begin{definition}\cite[Prop. 1.2]{cd}
\label{commxa} 
For $R \in \Com$ and $p$ a prime number, we have
\begin{enumerate}
\item (Verschiebung) $V: R^{\N_0}\to R^{\N_0}$ given by 
$ V(r_0,r_1,...) := p(0, r_0,r_1,...)$.
\item (Teichm\"{u}ller map) For $r\in R$ define 
$ \ab{r}:= (r,r^p,r^{p^2},...)$
\item $X(R)$ to be the closed subgroup of $R^{\N_0}$ generated by $\{V^n\ab{r} \ | \ n\in \N_0, \ r\in R\}$. Here the topology on $R^{\N_0}$ is the product topology. Note that $X(R)$ is invariant under $V$. 
\end{enumerate}

\noindent \label{commea} Let $A:=\Z [R]$ be the commutative polynomial ring on the set $R$. For $x \in R$ we denote the corresponding variable in $A$ by $[x]$. Let $$E(R):= \frac{X(A)}{X_I(A)}$$ where $I$ is the kernel of the natural surjection $A\to R$ and  $X_I(A) \subset X(A)$ is the closed subgroup generated by $$\{V^n\ab{a} - V^n\ab{b} \ | \ n \in \N_0,\ a, b \in A \text{ such that }a-b\in  I\}.$$
Thus we have a natural surjection $E(R):= \frac{X(A)}{X_I(A)} \to X(R)$.
\noindent $E(R)$ is naturally endowed with an endomorphism $V$ and a map $\ab{ \ } :R \to E(R)$ which sends $x$ to the class of $\ab{[x]}$ in $E(R)$. 
\end{definition}

\begin{remark}\label{rem:xi}
By Lemma \ref{everyxa} below, it is easy to see that the closed subgroup $X(I)$ generated by $\{V^n\ab{a} \ | \ n \in \N_0,  \ a\in  I\}$ as given in \cite[Prop.1.2]{cd} coincides with the closed subgroup $X_I(A)$ generated by 
$\{V^n\ab{a}-V^n\ab{b} \ |a,\ b \in A \text{ such that } a-b\in  I, \ n \in \N_0\}$. Thus the group $E(R)$ defined above for commutative rings is exactly the same as defined in \cite{cd}.
The subgroups $X(I)$ and $X_I(A)$ will not coincide for a non-commutative ring. For example, take a non-commutative polynomial ring $A:= \Z\{X, Y\}$ and $I:=\ab{X, Y} $ then $\ab{X} - \ab{Y}$ is in $X_I(A)$ but not $X(I)$. 
\end{remark}

\begin{lemma}\label{ilimxa} 
For $R \in \Com$, let $\tilde{X}(R)\subset X(R)$ be the subgroup generated by \\
$\ac{V^n\ab{r} \ | \ n\in \N_0, \ r\in R}$. Then 
\begin{enumerate}
\item The natural map $X(R) \to \ilim_n \frac{X(R)}{V^n(X(R))}$ is an isomorphism. 
\item The natural map $\frac{\tilde{X}(R)}{V^n(\tilde{X}(R))} \to \frac{X(R)}{V^n(X(R))}$ is an isomorphism.
\end{enumerate}
In particular we have a canonical isomorphism 
$X(R) \cong \ilim_n\frac{\tilde{X}(R)}{V^n(\tilde{X}(R))}$.
\end{lemma}

\begin{proof} To prove (1), it suffices to observe that, for any $n \in \N_0$,
$$ (\{0\}^n \times R \times R \times \cdot \cdot \cdot ) \cap X(R) = V^n(X(R)).$$ 
The surjectivity and statement (2) are deduced from Lemma \eqref{everyxa}. 
\end{proof}

\begin{lemma}\label{xisprewitt} $X : \Com \to \mathsf{Ab}$ is a pre-Witt functor.
\end{lemma}
\begin{proof} It is clear from the construction \ref{commea} and Lemma \ref{ilimxa} that,  $X(R)$ satisfies (1)-(3) of \ref{properties} for $R \in \Com$. It satisfies (4), follows from the fact that $X(R)$ is a subgroup of $R^{\N_0}$. Thus $X$ is a pre-Witt functor.
\end{proof}

\noindent We now recall an important result proved by Cuntz and Deninger, which shows that the property (4) of \ref{properties} is satisfied by $E(R)$.
\begin{lemma}\cite[Cor. 2.10]{cd}\label{property4}
$R$ is $p$-torsion free $\implies E(R) = X(R)$. In particular, $E(R)$ is $p$-torsion free. 
\end{lemma}
The following lemma, is an easy consequence of the above result. 
\begin{lemma}\label{everyxa}
Every element of $X(R)$ can be written as $\sum_{n=0}^\infty V^n\ab{a_n}$. 
\end{lemma}
\begin{proof}
Since $X(R)$ is the quotient of $X(\Z[R])$, it suffices to prove the statement when $R$ is a polynomial ring. When $R$ is a polynomial ring, $X(R) = E(R)$ \cite[Cor. 2.10]{cd} and $E(R)\cong W(R)$ \cite[page 561]{cd}. The result then follows from the corresponding well known bijection $R^{\N_0}\to W(R)$ which is indeed given by the formula $$(a_0, a_1,...) \mapsto \sum_{n} V^n\ab{a_n}$$ 
\end{proof}

\begin{lemma}\label{ertopgen}
The group $E(R)$ is topologically generated by $\{V^n\ab{x}|\ x\in R \ ,n\geq 0\}$. Moreover, the functor $E: \Com \to Ab$ satisfies the properties in \ref{properties} and hence is a pre-Witt functor.
\end{lemma}

\begin{proof}
Since, $X(\Z[R])$ is topologically generated by $\{V^n\ab{a} \ | a\in \Z[R], n\geq 0\}$ and $E(R)$ is quotient of $X(\Z[R])$, $E(R)$ is topologically generated by images of $\{V^n\ab{a} \ | \ a \in \Z[R] \ ,n\geq 0\}$ under the map $X(\Z[R]) \to  E(R)$. This proves the first part.
    
\noindent For the second part, by Lemma \ref{ilimxa}, $X(\Z[R])$ is an abelian group which is complete and Hausdorff w.r.t the $V$-filtration and  $X_I(\Z[R])$ being closed subgroup, $E(R)$ is $V$-complete by \cite[chap IX, §3.1, prop. 4]{bour}. This proves the $(3)$ in properties \ref{properties}. The fourth property follows from Lemma \ref{property4}. Other properties in \ref{properties} are easy to deduce.
\end{proof}
\vspace{2mm}

\section{The functor $C: \Com \to \mathsf{Ab}$}

\noindent Fix a prime $p \neq 2$ throughout this section. In this section, we first define a functor $C : \Com \to \mathsf{Ab}$ which attempts to enforce the properties \ref{properties}. It is less straightforward to see that $C$ satisfies \ref{properties}(4) and hence we introduce a temporary notion of a weak pre-Witt functor (see Definition \ref{defweakpre}). We will in fact show in the next section that $C$ is a pre-Witt functor (see Theorem \ref{cise}). In this section, however, we limit ourselves to showing that $C$ is a universal weak pre-Witt functor (see Theorem \ref{universalweakprewitt}).
\begin{definition}[The functor $C$]\label{functorc}
For any associative ring $R$ and a prime $p\neq 2$, we first define a group $G(R)$ as follows. Let 
\begin{enumerate}
\item $\tilde{G}(R)$ be the free abelian group generated by symbols $\{ V^n_r \ | \ n\in \N_0 , \ 0\neq r\in R\}$. Define $V^n_0:=0$ for all $n\geq0$. 
\item A set map $R \xrightarrow{\ab{\ } } \tilde{G}(R)$ given by $\ab{r} := V^0_r$.
\item A homomorphism $\tilde{G}(R) \xrightarrow{V} \tilde{G}(R)$ defined by $V(V^n_r):= V^{n+1}_r$.
\item $H(R)\subset \tilde{G}(R)$ be the subgroup generated by the set: 
$$  \Union_n\left(\{(V^n_{(x+y)^p} - p V^{n-1}_{x+y}) - (V^n_{x^p} - p V^{n-1}_{x}) - (V^n_{y^p} - p V^{n-1}_{y})\ | \ x, y \in R \} \Union \{V^n_r+V^n_{-r} \ | \ r \in R\}\right)$$ 
\item $\tilde{H}(R)$ be the $p$-saturation of $H(R)$, i.e.
$ \tilde{H}(R):= \{ \alpha \in \tilde{G}(R) \ | \ p^\ell \alpha \in H(R) \ \text{for some } \ell >0\}$.
\item Let $G^0(R)$ denote the completion of $\tilde{G}(R)/\tilde{H}(R)$ by the $V$-filtration. Let $G(R)$ be the quotient of $G^0(R)$ modulo the closed subgroup generated by $p$-power torsion elements. Note that the construction of $G(R)$ is functorial in $R$.\\
\end{enumerate}
Let $R\in \Com$ and 
$$ 0 \to I \to \Z[R] \to R\to 0$$
be the canonical free presentation of $R$.
Define $$C(R):=\frac{G(\Z[R])}{G_I(\Z[R])}$$ where $G_I(\Z[R])$ is the closed subgroup generated by 
$$\{V_x^n - V_y^n \ | \ x, y \in \Z[R] \text{ such that } x-y \in I, \ n \geq 0 \}.$$
We have a well defined group homomorphism $V: C(R) \to C(R)$ and a well defined map of sets $R \xrightarrow{\ab{ \ }} C(R)$ induced by the corresponding maps on $G(\Z[R])$.\\

\noindent Any ring homomorphism $f: R \to S,$ gives a group homomorphism $G(R) \to G(S)$ induced by $V^n_r \mapsto V^n_{f(r)}$ which is compatible with $V$ and hence a continuous group homomorphism $C(f): C(R) \to C(S).$ 
\end{definition}

\begin{remark}
There is a straight forward modification of the definition of $C$ for $p=2$. However we only restrict ourselves to the condition $p\neq 2$ due to it's requirement in Lemma \ref{zindcomm}. We have not yet explored the case $p=2$.
\end{remark}

\noindent The next result is about $C(A)$ when $A$ is a polynomial algebra. 
\begin{lemma}\label{ceqg}
If $A=\Z[S]$ is a free commutative algebra, then $C(A) \cong G(A)$. In particular, $C(A)$ is $p$-torsion free.
\end{lemma}
\begin{proof}
Since $A$ is a free algebra the natural surjection $\pi: \Z[A] \to A$ has a section which is also a ring homomorphism. We denote this section as $j : A \to \Z[A]$. Then we have a commutative diagram 

\[
\begin{tikzcd}
  0 \arrow[r] & 0 \arrow[d] \arrow[r] & A \arrow[d, "j"] \arrow[r, "id"] & A \arrow[d, "id_A"] \arrow[r] & 0 \\
  0 \arrow[r] & I \arrow[r] & \Z[A] \arrow[r, "\pi"] & A \ar[r] & 0.
\end{tikzcd}
\]

\noindent By functoriality of $G$, we have a group homomorphism $$G(j): G(A) \to G(\Z[A])$$ 
Composing with the quotient map, we get a morphism $$\phi: G(A) \to C(A)$$
It suffices to show that $\phi$ is an isomorphism.  The injectivity of $\phi$ is straight forward because $\pi \circ j = id_A$ and hence $\phi$ has a retract induced by $G(\pi)$. For surjectivity we note that for $\alpha \in G(\Z[A])$, $\phi(\beta)=\alpha$ where $\beta:=G(\pi)(\alpha)$. 
\end{proof}

\noindent We now state a very important property of the functor $C$, which we will prove in the next section.
\begin{theorem} For $p\neq 2$, $C$ is a universal pre-Witt functor. 
\end{theorem}
\noindent Unfortunately, it appears less straightforward to check that $C$ actually satisfies the property \ref{properties}(4) and is a pre-Witt functor. To overcome this hurdle, we introduce a notion of `weak pre-Witt' functor below, whose role is somewhat temporary (see Theorem \ref{cise}).
\begin{definition}[weak pre-Witt functor]\label{defweakpre} A functor $F: \Com \to  \mathsf{Ab}$ is said to be a weak pre-Witt functor if it satisfies the properties (1), (2) and (3) of \ref{properties} and the following property, which is a weaker version of  \ref{properties}(4): \\

$(4')$ $A$ is  free (i.e. a polynomial ring)  $\implies F(A)$ is $p$-torsion free.
\end{definition} 
\vspace{1mm}
\noindent Any pre-Witt functor is a weak pre-Witt functor. In particular, the functors $W$ and $E$ are weak pre-Witt functors. 
 
\begin{theorem}\label{universalweakprewitt} Let $p\neq 2$. The functor $C: \Com \to Ab$ is a universal weak pre-Witt functor. In particular, there is a natural transformation $C\xrightarrow{\eta} E$ which is compatible with $V, \ab{ \ }$.
\end{theorem}
\begin{proof} It follows from the construction of $C$ that, $C$ satisfies (1), (2) and (3) of \ref{properties}. Lemma \ref{ceqg} implies that the property $(4')$ of the Definition \ref{defweakpre} is also satisfied by $C$. Hence $C$ is a weak pre-Witt functor. It remains to show its universality. Let $F$ be any weak pre-Witt functor. For a commutative ring $R$, take the presentation $0 \to I \to \Z[R] \xrightarrow{\pi} R \to 0$. We define the homomorphism, 

$$\eta: \tilde{G}(\Z[R]) \to F(\Z[R]), \ \text{given by} \  V^n_{a}  \mapsto V^n\ab{a} \ \text{for all} \  a \in \Z[R], \ n \geq 0.$$ 

\vspace{1mm}
\noindent Since $F$ satisfies properties (1)-(2) of \ref{properties}, $\eta(\tilde{H}(\Z[R]))=0$ in $F(\Z[R])$. So, we get a homomorphism (again denoted by $\eta$),

 $$\eta:\tilde{G}(\Z[R])/\tilde{H}(\Z[R]) \to F(\Z[R]).$$ 
 
\vspace{1mm}
\noindent As $\eta$ is compatible with $V$, and $F$ satisfies \ref{properties}(3), taking $V$-completion, we get the transformation 
 $\eta: G^0(\Z[R]) \to F(\Z[R])$. Since $F(\Z[R])$ is $p$-torsion free, $\eta$ induces a map, which we continue to denote by $\eta : G(\Z[R])\to F(\Z[R])$. Now, consider the composition  $F(\pi) \circ \eta$, where $F(\pi): F(\Z[R]) \to F(R)$. This composition, still denoted by $\eta$, factors through $G_I(\Z[R])$. Indeed, for $a, b \in \Z[R]$ such that $a-b \in I, \eta(V^n_a-V^n_b)= V^n\ab{\pi(a)}-V^n\ab{\pi(b)} = 0$, so $\eta(G_I(\Z[R])) = 0$ in $F(R).$ Therefore, we get the natural transformation $\eta: C(R) \to F(R).$ Hence, $C$ is universal weak pre-Witt functor.
 \end{proof}
\vspace{1mm}
\section{Proof of main theorem \ref{univcom}}

\noindent In this section we will prove the main theorem \ref{univcom}. Note that, for $R \in \Com$, we have the canonical isomorphism $E(R) \to W(R)$ by \cite[Page 561]{cd}. Thus, in view of the above Theorem \ref{universalweakprewitt}, it is enough to prove that $C(R)$ is a pre-Witt functor and that $C(R) \cong E(R)$ for all $R \in \Com$ (see Theorem \ref{cise}).
The first step towards proving this is to show that, for a free commutative polynomial algebra $A$, $C(A) \cong X(A)$. The crucial observation for this is the Lemma \ref{zindcomm}. Before proving the Lemma \ref{zindcomm}, we will begin by stating an interesting fact about integers. 
\begin{lemma}\label{genvm}
Let $p$ be any prime and $\{c_1,c_2,...,c_n\}$ be distinct nonzero integers. 
Let
$$M(c_1,c_2,...c_n):= \left[\begin{matrix}
c_1 & c_2 & ... & c_n\\
c_1^p & c_2^p & ... & c_n^p\\
\vdots   \\
c_1^{p^{n-1}} & c_2^{p^{n-1}} & ... & c_n^{p^{n-1}}
\end{matrix}\right]$$
Then  columns of $M(c_1,...,c_n)$ are $\Z$-linearly independent in the following cases
\begin{enumerate}
\item $c_i$ are positive integers for  all $i$. 
\item $p\neq 2$ and $|c_i|$ are distinct, i.e. $c_i\neq -c_j$ for any $i\neq j$
\end{enumerate}
\end{lemma}

\begin{proof} Clearly columns of $M(c_1,...,c_n)$ are $\Z$-linearly independent iff $det(M(c_1,...,c_n))\neq 0$.
The first statement is a well known fact and follows from the exercise as given in \cite[page 43]{ps}. The second statement can be deduced from (1) as follows. Since $p$ is odd, $det(M(c_1,c_2,..,c_n))$ differs from $det(M(|c_1|,|c_2|,...,|c_n|)$ at most by a sign. The condition in (2) ensures that $|c_i|$ are distinct and now apply (1).
\end{proof}

\begin{lemma}\label{zindcomm} Let $p\neq 2$.
Let $A = \Z[S]$ be a commutative polynomial ring over a set $S$. Let $\{f_i\}_{i=1}^r$ be a finite set of distinct non-zero elements of $A$. Further assume that $f_i\neq -f_j$ for any $i\neq j$. Then the subset $\{\ab{f_i}\}_{i=1}^r$ of $X(A)$ is $\Z$-linearly independent. 
\end{lemma}
\begin{proof}
As there are only finitely many variables in each $f_i$, we may assume $A=\Z[X_1,...,X_n]$. Since $\Z$ is an infinite integral domain, it is well known that a polynomial $f \in A$ is zero iff $f(\underline{a})=0$ for all $\underline{a}\in \Z^n$. Applying this to the polynomial $g:=\Pi_{i<j\leq r}(f_i-f_j)(f_i+f_j)\cdot \Pi_i f_i$, which is non-zero by hypothesis, we can choose an element $\underline{a}\in \Z^n$ such that $g(\underline{a})\neq 0$. This implies that $c_i:=f_i(a)$ are distinct nonzero integers such that $c_i\neq -c_j$ for $i\neq j$. The linear independence of $\ab{f_i}$ now directly follows by Lemma \ref{genvm}
\end{proof}

\begin{lemma} \label{genrelxa} Let $A=\Z[S]$ be a free polynomial algebra. 
There exists a canonical isomorphism  $C(A) \cong X(A)$. 
\end{lemma}
\begin{proof}
We know that $X(A)$ is a pre-Witt functor by Lemma \ref{xisprewitt}. Since $C$ is a universal weak pre-Witt functor there exists a natural transformation $\eta: C \to X$. For a polynomial algebra $A$, we have the group homomorphism $\eta: C(A) \to X(A)$ which is given by $V^n_a \mapsto V^n\ab{a}$.  By Lemma \ref{ilimxa}(1) this map is surjective. By Lemma \ref{ceqg} we know that $G(A)\cong C(A)$, thus composing with the group homomorphism $\tilde{G}(A)/\tilde{H}(A) \to G(A)$, we get the map 
 
$$\overline{\eta}: \tilde{G}(A)/\tilde{H}(A) \to G(A) \to X(A)$$
The image of $\overline{\eta}$ lands in $\tilde{X}(A)$. Thus we have the group epimorphism 
$$\tilde{\eta}:  \tilde{G}(A)/\tilde{H}(A) \to \tilde{X}(A);  \  \ \ V^n_a \mapsto V^n\ab{a} .$$

\noindent We will now prove the theorem in following steps.\\

\noindent {\underline{Step 1}:} To prove the theorem it is enough to show that $\tilde{\eta}$ is an isomorphism: This is because, if $\tilde{\eta}$ is an isomorphism then since $\tilde{\eta}$ is compatible with $V$, after taking the $V$-completion we get that $$G^0(A) \cong X(A).$$ Since $X(A)$ has no $p$-torsion, this implies, using Lemma \ref{ceqg} $$C(A) \cong G(A) \cong X(A).$$

\noindent Thus, it is enough to show that the group epimorphism $\tilde{\eta}$ is injective.  This we will prove in the next step.\\

\noindent {\underline{Step 2}:} To show that $\tilde{\eta}$ is injective, we in fact consider the map (again denoted by $\tilde{\eta}$);
$$\tilde{\eta}: \tilde{G}(A) \to \tilde{X}(A)$$
We will show that $\Ker(\tilde{\eta}) = \tilde{H}(A)$. It is easy to observe from the definition of $\tilde{H}(A)$ that $\tilde{H}(A) \subseteq \Ker(\tilde{\eta})$. \\

Let  $\alpha= \sum_{i=1}^\ell c_i V^{n_i}_{a_i} \in \Ker(\tilde{\eta})$. We will prove that $\alpha\in \tilde{H}(A)$, i.e. 
$p^k \alpha \in H(A)$ for some $k \geq 0$. We will establish this in following two steps.\\

\noindent {\underline{Step 2a}:} We first prove the case when $n_i=0$ for all $i$, i.e. 
$$ \alpha = \sum_{i=1}^\ell c_iV^0_{a_i}$$ 
We may assume without loss of generality that $a_i$ are all distinct. Moreover if $a_r=-a_s$ for some $r\neq s$ then 
$$ V^0_{a_r} = - V^0_{a_s} \ {\rm  mod} \ H(A).$$ Thus we may replace the expression $c_rV^0_{a_r}+c_sV^0_{a_s}$ with  $(c_s-c_r)V^0_{a_s}$. In particular we may assume that not only $a_i$ are distinct, but $a_i\neq -a_j$ for any $i\neq j$. In this case, $c_i$ must be zero for all $i$ by Lemma \ref{zindcomm}. Thus $\alpha=0$ and hence is in $\tilde{H}(A)$. \\

\noindent {\underline{Step 2b}:} For $n_i \geq 0$, we prove the claim by induction on $n_{\alpha}: = {\rm max}_i\ac{n_i}- {\rm min}_i \ac{n_i}$. Since $\tilde{X}(A) \xrightarrow{V} \tilde{X}(A)$ and $\tilde{H}(A)\xrightarrow{V} \tilde{H}(A)$ are both injective maps, we may assume without loss of generality that ${\rm min}_i \ac{n_i} = 0$.  The starting step of induction is proved in Step 1. Now by rearranging, we may assume that $n_i = 0$ for $1\leq i \leq s$ and $n_i>0$ for $s < i\leq \ell$. Thus without loss of generality we have 
$$ \alpha = \sum_{i=1}^s c_iV^0_{a_i} + \sum_{i>s}^\ell c_iV^{n_i}_{a_i}$$

\noindent We need to show $p^k \alpha \in H(A)$, for some $k \geq 0$.
\begin{align*}
p\alpha & = \sum_{i=1}^s pc_iV^0_{a_i} + \sum_{i>s}^\ell pc_iV^{n_i}_{a_i} \\ 
       & = \sum_{i=1}^s c_i (pV^0_{a_i}-V^1_{a_i^p}) +\sum_{i=1}^s c_iV^1_{a_i^p} + \sum_{i>s}^\ell pc_iV^{n_i}_{a_i} \\
            & = \beta + \sum_{i=1}^s c_iV^1_{a_i^p} + \sum_{i>s}^\ell pc_iV^{n_i}_{a_i} \\
            & =: \beta + \gamma 
\end{align*}
Since $\alpha \in \Ker(\tilde{\eta})$ we have 
$$p\tilde{\eta}(\alpha) = \tilde{\eta}(\beta) + \tilde{\eta}(\gamma) = (p\sum_{i=1}^sc_ia_i, \cdots ) = (0, 0, 0, \cdots).$$
This implies that $\sum_{i=1}^s c_ia_i = 0$ since $X(A)$ is p-torsion free.\\

\noindent We now claim that $\beta \in H(A)$: To prove this claim we will use the defining relations of $H(A)$.\\
 Consider the map  $\phi: A \to \tilde{G}(A)$ given by $ a \mapsto V^1_{a^p}-pV^0_{a}$ which is a linear map, i.e. 
 $$\phi(\sum_{i=1}^s c_ia_i) \equiv \sum_{i=1}^sc_i \phi(a_i) (\text{mod} \  H(A))$$
 Thus, $\beta \equiv 0 (\text{mod} \  H(A))$. This proves that $\beta \in H(A)$ and hence is in $\tilde{H}(A)$. Thus, $\beta \in \Ker(\tilde{\eta})$. This implies that $\gamma \in \Ker(\tilde{\eta})$. Moreover observe that $n_{\gamma} < n_{\alpha}$, hence by the induction hypothesis, $\gamma \in \tilde{H}(A)$.  Thus, $p^{r}\gamma \in H(A)$ for some $r \geq 0$. Thus, $p^{r+1}\alpha \in H(A)$. 
 Thus, $\Ker(\tilde{\eta}) \subseteq \tilde{H}(A)$. \\ 
 
\noindent This proves that $\tilde{G}(A)/ \tilde{H}(A) \cong \tilde{X}(A)$ and hence the theorem. 
\end{proof}

\begin{theorem}\label{cise}
The natural transformation $C \to E$ is an isomorphism of functors. In particular, $C$ is a pre-Witt functor.
\end{theorem}
\begin{proof} Let $A$ be a commutative polynomial ring. Since $E(A) \cong X(A)$ by \cite[Cor. 2.10]{cd} and $X(A) \cong C(A)$ by Lemma \ref{genrelxa}, we get that $C(A) \cong E(A)$. For a general commutative ring $R$, consider the presentation, $$0 \to I \to \Z[R] \to R \to 0.$$
The functor $C$ being universal, we have a commutative diagram with exact rows.	
\[
\begin{tikzcd}
0 \arrow[r] & G_I(\Z[R]) \arrow[r] \arrow[d] & G(\Z[R]) \arrow[r] \arrow[d, "\phi"] & C(R) \arrow[r] \arrow[d] & 0 \\
0 \arrow[r] & X_I(\Z[R]) \arrow[r]           & X(\Z[R]) \arrow[r]           & E(R) \arrow[r]           & 0
\end{tikzcd}
\]

\noindent Now by Lemma \ref{ceqg} and Lemma \ref{genrelxa}, the verticle map $$\phi: G(\Z[R]) \to X(\Z[R])$$ is an isomorphism. The definition of $G_I(\Z[R])$ and $X_I(\Z[R])$ shows that the preimage of $X_I(\Z[R])$ under $\phi$ is exactly $G_I(\Z[R])$.  Hence, the induced map $C(R) \to E(R)$ is an isomorphism. 
\end{proof}
\vspace{1mm}

\begin{proof}[Proof of Theorem \ref{univcom}] By Theorem \ref{universalweakprewitt} and Theorem \ref{cise} we know that $C$ is a universal pre-Witt functor. By the canonical isomorphism $E(R) \to W(R)$ \cite[Page 561]{cd} and by Theorem \ref{cise}, $W$ is a universal functor satisfying four properties in \ref{properties}. 
\end{proof}
\vspace{1mm}

\end{document}